\documentclass[preprint,12pt]{elsarticle}
\usepackage{amsmath,amssymb,amsthm}
\usepackage{xcolor}
\usepackage{graphicx}
\usepackage{geometry}
\usepackage{caption}
\usepackage{hyperref}

\newtheorem{theorem}{Theorem}[section]
\newtheorem{lemma}[theorem]{Lemma}

\newtheorem{remark}{Remark}
\newtheorem{definition}{Definition}

\DeclareMathOperator*{\esup}{ess\ sup}

\journal{...}

\begin{document}

\begin{frontmatter}

\title{Existence results for the Cox–Ingersoll–Ross model  \\ with variable exponent diffusion}
\author{Mustafa Avci}

\affiliation{organization={Faculty of Science and Technology, Applied Mathematics\\  Athabasca University},%Department and Organization
            city={Athabasca},
            postcode={T9S 3A3},
            state={AB},
            country={Canada}}
\ead{mavci@athabascau.ca (primary) & avcixmustafa@gmail.com}

\begin{abstract}
\noindent This study proposes a new  stochastic model where the diffusion coefficient involves a state-dependent variable exponent function $p(\cdot)$. This new theoretically flexible framework generalizes the classical Cox–Ingersoll–Ross model. The existence, uniqueness, and higher moment properties of solutions are analyzed. The validity and efficiency of the model is illustrated with numerical experiments. Finally, a detailed analysis for It\^o vs Stratonovich interpretations of the proposed model is provided.
\end{abstract}


\end{frontmatter}

\section{Introduction}

\label{sec1}
In this paper, we study the existence and uniqueness of solutions for the following nonlinear stochastic differential equation with variable exponent in the diffusion
\begin{align} \label{eq.1}
dv(t) &=\kappa (\theta-v(t))dt + \xi v(t)^{p(v(t))}dW(t),\,\, t\in [0,T]
\end{align}
which models the evolution of the state variable $v(t)$. By the definition of stochastic differential, the equation \eqref{eq.1} is equivalent to the following stochastic integral equation
\begin{equation} \label{eq.2}
v(t) = v_0 + \int_0^t \kappa (\theta-v(s)) ds + \int_0^t \xi v(s)^{p(v(s))}dW(s),
\end{equation}
where the initial condition $v(0) = v_0$ is a random variable; $W(t)$ is a standard Brownian motion; $\kappa, \theta, \xi> 0$ are real parameters; $p(\cdot)$ is a function of the state variable $v(t)$ satisfying some hypotheses.\\

The nonlinear stochastic differential equation \eqref{eq.1} can be considered as a generalized version of the Cox–Ingersoll–Ross (CIR) model \cite{cir85}. In the CIR model, the variance (diffusion) process $v(t)=\{v(t):\, t\geq 0\}$ evolves by
\begin{align} \label{eq.1a}
\mathrm{(CIR)}
\quad
dv(t) &=\kappa (\theta-v(t))dt + \xi \sqrt{v(t)}dW(t),
\end{align}
where $\kappa > 0$ is the speed of mean reversion, $\theta>0$ is the long‐run level, and $\xi>0$ is the volatility parameter.  To generalize \eqref{eq.1a}, we replace $\sqrt{v(t)}$ by $v(t)^{p(v(t))}$ for a state‐dependent exponent function $p:(0,\infty)\to\mathbb{R}$ and obtain the \textit{generalized CIR model (GM)}, i.e. the nonlinear stochastic differential equation \eqref{eq.1}.\\

The CIR model is first used to study the interest rate dynamics and then used in the Heston model
to describe the stochastic volatility \cite{heston93}. Since then it has been used in many different models involving credit‐risk and default intensity models, population biology and ecological models, and diffusion in heterogeneous media to name a few.\\

Due to its strong relevancy, it is necessary to mention the paper \cite{pkm10} in which the authors use various samples of S\&P 500 index return data and compare the empirical performance of the
Heston model with five simple alternatives to maximise model fit for the samples, all of which can be described by
\begin{equation}\label{eq.1b}
dv(t) = \kappa v^{a}(t)(\theta-v(t))dt + \xi v(t)^{b}dW(t),\,\, t \in [0,T],
\end{equation}
for $a=\{0,1\},\, b=\{1/2,1,3/2\}$, we call these $\{a, b\}$ values as $(a,b)$-family.\\
The relationship between our generalized model and \eqref{eq.1b} is immediate and can be interpreted as follows:
\begin{itemize}
  \item If $p(v)\equiv b \in \{1/2,1\}$ is constant for all $v$ and $a=0$ in \eqref{eq.1b}, then GM coincides with the corresponding member of the $(a,b)$-family; that is, GM is algebraically identical to the $(a,b)$-member with exponent $b$. In particular $p\equiv 1/2$ recovers the classical CIR case.
  \item If $p(\cdot)$ is not constant (the case studied in this paper), then GM and \eqref{eq.1b} are different models. GM permits a state-dependent diffusion exponent, and hence requires a state-dependent analysis, while \eqref{eq.1b} describes fixed algebraic powers for the diffusion.\\
\end{itemize}
We also would like to mention that extensions of the CIR framework to more flexible dynamics have been studied in various directions (time-changed intensities, switching regimes, etc.) (see e.g., \cite{ml14}).  From an analytical perspective, pathwise uniqueness and strong-solution techniques for SDEs with non-Lipschitz coefficients have been refined in recent works (see, e.g., \cite{s17}), and rigorous positivity-preserving approximation results for CIR-type dynamics are available (see, e.g., \cite{cr18}).\\
Motivated mainly by the aforementioned works, we introduce the nonlinear stochastic differential equation \eqref{eq.1} as a natural generalization of the classical CIR model.
The primary motivation for replacing the classical square-root diffusion $\sqrt{v(t)}$ by a state-dependent power $v(t)^{p(v(t))}$ is to allow the local noise intensity to react flexibly to the current state of the process where such flexibility can capture regime-dependent amplification of fluctuations; for example, stronger noise when the variance is large or a permanently elevated baseline of noise during crisis regimes.  However, this generalization creates two interrelated mathematical and numerical difficulties that we address in this work.\\
First, the map $v\mapsto v^{p(v)}$ typically fails to be uniformly Lipschitz on $(0,\infty)$ and may violate standard linear growth conditions.  In particular, the classical Picard–Lindelöf (or standard SDE) framework and many elementary moment estimates no longer apply since near zero the nonlinearity $v^{p(v)}$ can be weaker or stronger than the square-root, and globally the rate of growth of the diffusion coefficient depends on the path through $p(\cdot)$.  These features complicate the construction of a global strong solution,  the derivation of higher-order moment bounds, and positivity arguments which are common for the CIR model. To overcome these difficulties, we approximate the problematic coefficient $v\mapsto v^{p(v)}$ by a family of uniformly Lipschitz functions (obtained by smoothing and truncation techniques near problematic regions).  This yields local solutions via standard Picard iterations.  To extend local solutions to a global, strictly positive strong solution we combine a truncation (cut-off) procedure with a successive-approximation argument and employ measure-theoretic tools to pass to the limit.  Uniqueness and the strong-solution property are then established using the standard existence theorems (Lemma \ref{Lem:2.1}). This two‐step procedure establishes both the existence and uniqueness of a strictly positive strong solution to \eqref{eq.1}.\\
Second, as well-known, non-uniform Lipschitzity weakens the usual convergence guarantees for standard numerical schemes; for example, explicit Euler–Maruyama can generate false negative values or large discretization noise when not applied carefully. To handle this, in our simulations we use the Euler–Maruyama scheme with two practical stabilizations: (i) a positivity truncation to avoid false negative iterates, and (ii) a sufficiently small time step together with parameter sensitivity checks to verify robustness.  As well, Monte Carlo comparisons are performed with perfectly coupled Brownian increments (the same $\Delta W_k$ for each paired simulation), which isolates the effect of the exponent function $p(\cdot)$ from sampling variability.\\

The paper is organized as follows. In Section \ref{Sec2}, we first provide background on measure-theoretic probability and the theory of stochastic process. Then we obtain an auxiliary result, Lemma \ref{Lem:2.3}, which ensures the well-definedness of \eqref{eq.1}. In Section \ref{Sec3}, we prove the existence and uniqueness of solutions. In Section \ref{Sec4}, we obtain an estimate on the higher moments of the solution. In Section \ref{Sec5}, to confirm our model and theoretical results, we conduct numerical experiments.
In Appendix, we provide a detailed analysis for It\^o vs Stratonovich interpretations of our model.

\section{Preliminaries and auxiliary results} \label{Sec2}

We start with some basic concepts of the measure-theoretic probability and the theory of stochastic process. For more details, we refer the reader to, e.g., \cite{bz98,ovidiu22,friedman75,mao07,oksendal03}.\\
\noindent Let $W(t),\, t\geq 0$ be a Brownian motion on a probability space $(\Omega, \mathcal{F}, \mathbb{P})$ under the real-world measure $\mathbb{P}$. Let $\mathcal{F}_t$, $t\geq 0$, be an increasing family of $\sigma$-fields, called a \emph{filtration}, denoted by $\{\mathcal{F}_t\}_{t \geq 0}$. Then $(\Omega, \mathcal{F}, \{\mathcal{F}_t\}_{t \geq 0}, \mathbb{P})$ is called a \emph{filtered probability space}.\\
A \emph{stochastic process} is an indexed collection of random variables $X(t)=\{X(t):\, t_0\leq t \leq T\}$, with $0\leq t_0 < T< \infty$. $X(t)$ is said to be \emph{adapted} to $\mathcal{F}_{t}$ if for each $t \in [t_0,T]$, $X(t)$ is $\mathcal{F}_{t}$-measurable.\\
For every $\omega \in \Omega$, the function $t\longmapsto X(t,\omega)$ is called a \emph{path} (or \emph{sample path}) of $X(t)$.\\

\noindent We denote by $\mathcal{L}^{q}[t_0,T]$ ($1\leq q \leq \infty$) the class of real-valued $\mathcal{F}_{t}$-adapted stochastic processes $X(t)$  satisfying:
\begin{align*}
\mathbb{P}\left\{\int_{t_0}^{T}|X(t)|^{q}dt<\infty \right\}=1 \, \text{ if }\, 1\leq q < \infty \quad \text{and} \quad
\mathbb{P}\left\{\esup_{t_0 \leq t \leq T}|X(t)|<\infty \right\}=1\, \text{ if }\, q=\infty.
\end{align*}
By $\mathcal{M}^{q}[t_0,T]$ we denote the subset of $\mathcal{L}^{q}[t_0,T]$ consisting of all stochastic processes $X(t)$ satisfying:
\begin{equation*}
\mathbb{E}\int_{t_0}^{T}|X(t)|^{q}dt<\infty \, \text{ if }\, 1\leq q < \infty \quad \text{and} \quad \mathbb{E}\left(\esup_{t_0 \leq t \leq T}|X(t)| \right)<\infty\, \text{ if }\, q=\infty.
\end{equation*}
Consider the $1$-dimensional stochastic differential equation of It\^{o} type
\begin{equation}\label{eq.4a}
dX(t) = \mu(X(t),t)dt + \sigma(X(t),t)dW(t), \quad t\in[t_0,T],
\end{equation}
with the initial value $X(t_0)=X_0$, where $X_0 \in \mathbb{R}$ is an  $\mathcal{F}_{t_0}$-measurable random variable such that $\mathbb{E}|X_0|^2<\infty$.\\
The equation \eqref{eq.4a} is equivalent to the following stochastic integral equation
\begin{equation}\label{eq.4b}
X(t) = X_0 + \int_{t_0}^{t} \mu(X(s),s)ds + \int_{t_0}^{t} \sigma(X(s),s)dW(s), \quad t\in[t_0,T],
\end{equation}
where $\mu:\mathbb{R}\times[t_0,T]\to \mathbb{R}$ and $\sigma:\mathbb{R}\times[t_0,T]\to \mathbb{R}$ are both Borel measurable functions.\\
\begin{definition}\label{Def:2.1}
A real-valued stochastic process $X(t)$ is called a solution of equation \eqref{eq.4a} if it has the following properties:
\begin{itemize}
  \item [$(i)$] $X(t)$ is continuous and $\mathcal{F}_{t}$-adapted;
  \item [$(ii)$] $\mu(X(t),t) \in \mathcal{L}^{1}[t_0,T]$ and $\sigma(X(t),t) \in \mathcal{L}^{2}[t_0,T]$;
  \item [$(iii)$] equation \eqref{eq.4b} holds for every $t\in[t_0,T]$ with probability 1.
\end{itemize}
A solution $X(t)$ is said to be unique if any other solution $\hat{X}(t)$ is indistinguishable from $X(t)$; that is,
$$
\mathbb{P}\left\{X(t) =\hat{X}(t),\,\, t_0 \leq t \leq T\right\}=1.
$$
\end{definition}
In the sequel, the symbol $\vee$ denotes $a \vee b = \max\{a, b\}$.
\begin{lemma}\label{Lem:2.1}\cite{mao07}
Assume that there exist two positive constants $L$ and $K$ such that
\begin{itemize}
  \item [$(i)$] (Lipschitz condition) for all $x,y \in \mathbb{R}$ and $t \in [t_0,T]$
   \begin{equation}\label{eq.38g}
    |\mu(x,t)-\mu(y,t)|^2 \vee |\sigma(x,t)-\sigma(x,y)|^2 \leq L|x-y|^2;
   \end{equation}
   \item [$(ii)$] (Linear growth condition) for all $x\in \mathbb{R}$ and $ t \in [t_0,T]$
   \begin{equation}\label{eq.38h}
   |\mu(x,t)|^2 \vee |\sigma(x,t)|^2 \leq K(1+|x|^{2}).
   \end{equation}
\end{itemize}
Then there exists a unique solution $X(t)$ to equation \eqref{eq.4a} and $X(t) \in \mathcal{M}^{2}[t_0,T]$. Further,
\begin{equation}\label{eq.39a}
\mathbb{E}\sup_{t_0 \leq t \leq T}|X(t)|^{2} \leq (1+ 3\mathbb{E}|X_0|^2)\exp(3 K(T-t_0)(T-t_0+4)).
\end{equation}
\end{lemma}
\noindent Throughout the paper, we assume that the variable exponent function $p(\cdot)$ satisfies the following hypotheses.
\begin{itemize}
   \item [$(\mathbf{p_1})$] $p(\cdot): [0, \infty) \to \mathbb{R}$ is a differentiable function satisfying
   \begin{align}\label{eq.3}
     1/2 &\leq p^-:=\inf_{x\geq 0} p(x) \quad \text{and} \quad \sup_{x\geq 0} p(x):=p^+ \leq 1.
   \end{align}
   \item [$(\mathbf{p_2})$] There exists a real number $\delta$ such that
   \begin{align}\label{eq.4}
     \sup_{0<x<\delta} |p^{\prime}(x)|<\infty.
   \end{align}
\end{itemize}
In the rest of the paper, we let $f(x):=\kappa (\theta-x)$ and $g(x):=\xi x^{p(x)}$ for $x \in (0,\infty)$.\\

\noindent The following lemma makes sure that the equation \eqref{eq.1} is well-defined.
\begin{lemma}\label{Lem:2.3}
Assume that $\kappa, \theta, \xi, v_0 > 0$, and $(\mathbf{p_1})$-$(\mathbf{p_2})$ hold. Then the process $v(t)$ is strictly positive for $t \in [0,T]$.
\end{lemma}

\begin{proof}
We use the following Feller’s test \cite{feller52,feller51}, which is also know as the local (differential) form of Feller’s non-attainability test, to analyze the behavior of the diffusion processes $v(t)$ at the boundary (i.e. $v(t)=0$) of its state space $(0,\infty)$.\\
Consider equation \eqref{eq.4a}. If the condition
\begin{equation}\label{eq.38b}
\lim_{x \to 0}\left(\mu(x)-\frac{1}{2}\frac{\partial \sigma^2}{\partial x}(x) \right)\geq 0
\end{equation}
holds, then the boundary $X(t)=0$ is non-attainable for the process $X(t)$ if  $X_0>0$.\\
Now, we apply \eqref{eq.38b} to the diffusion process \eqref{eq.1}.\\
First define
\begin{equation}\label{eq.38c}
\mathcal{T}_p(x)=f(x)-\frac{1}{2}\frac{\partial g^2}{\partial x}(x).
\end{equation}
Then
\begin{align}\label{eq.38d}
\mathcal{T}_{p}(x)&=\kappa(\theta - x)-\xi^2 x^{2p(x)}\left(p^{\prime}(x)\log(x)+\frac{p(x)}{x}\right)\\
& \geq \kappa(\theta - x)-\xi^2 \left(|p^{\prime}(x)||\log(x)|x^{2p(x)}+x^{2p(x)-1}p(x)\right),\nonumber
\end{align}
for $0<x<\delta$. Now applying the elementary limit theorems and the assumptions we have, it reads
\begin{align}\label{eq.38e}
\lim_{x \to 0^+}\mathcal{T}_{p}(x)& \geq \lim_{x \to 0^+} \left(\kappa(\theta - x)-\xi^2 \left(|p^{\prime}(x)||\log(x)|x^{2p(x)}+x^{2p(x)-1}p(x)\right)\right)\nonumber\\
&=\kappa \theta >0,
\end{align}
as desired. In conclusion, the diffusion process $v(t)$ started in $(0,\infty)$ never hits zero almost surely (a.s.), and hence, the equation \eqref{eq.1} is well-defined.
\end{proof}

\begin{remark}\label{Rem:2.1}
If we let $p(x)=1/2$ in \eqref{eq.1}, then
\begin{align}\label{eq.38f}
\mathcal{T}_{1/2}(x)=\kappa(\theta - x)-\xi^2 x\left(\frac{1}{2x}\right) \Rightarrow \lim_{x \to 0^+}\mathcal{T}_{1/2}(x) \geq \kappa \theta-\frac{\xi^2}{2} \geq 0,
\end{align}
which gives the condition  $2 \kappa \theta\geq \xi^2$, i.e. the well-know Feller's (zero) non-attainability condition for the CIR model.
\end{remark}

\section{Existence and uniqueness results} \label{Sec3}
The following theorem is the main result of the present paper.

\begin{theorem}\label{Thrm:3.1}
Assume that conditions $(\mathbf{p_1})$-$(\mathbf{p_2})$ hold. Suppose that the linear growth condition $(ii)$ of Lemma \ref{Lem:2.1} holds; however, the Lipschitz condition is replaced with the following local Lipschitz condition:\\
For every integer $n \geq 1$, there exist positive constants $\hat{L}_n$ such that for all $t \in [0,T]$ and for all $ x,y \in \mathbb{R}$ with $|x|,\, |y| \in [\frac{1}{n},n]$ it holds:
\begin{equation}\label{eq.3.7}
|f(x)-f(y)|^2 \vee |g(x)-g(y)|^2  \leq \hat{L}_n\,|x-y|^2.
\end{equation}
Let $v(0) = v_0>0$ be a random variable independent of $\mathcal{F}_t$ such that $\mathbb{E}|v_0|^2<\infty$. Then there exists a unique solution $v(t)$ of \eqref{eq.1} in $\mathcal{M}^{2}[0,T]$.\\
\end{theorem}
\begin{proof} We first start with a truncation procedure.\\
For each integer $n \ge 1$, we choose a small positive parameter $\varepsilon=\varepsilon(n)$ such that $\varepsilon< 1/n^2$. Next, we define a piecewise-function $\theta_n(\cdot): [0, \infty) \to \mathbb{R}$  by
\begin{equation}\label{eq.5}
\theta_n(r) = \begin{cases} \frac{1}{n}, & 0 \le r \le \frac{1}{n} \\ r, & \frac{1}{n} + \varepsilon \le r \le n - \varepsilon \\ n, & r \ge n \end{cases}
\end{equation}
We also define the smooth radial truncation $\rho_n^{\varepsilon}(x)$ for $x \in \mathbb{R}$ by
\begin{equation} \label{eq.6}
\rho_n^{\varepsilon}(x) = \begin{cases} \theta_n(|x|)\operatorname{sgn}(x), & x \neq 0 \\ 0, & x = 0 \end{cases}
\end{equation}
Defined this way, it holds $|\rho_n^{\varepsilon}(x)|=|\theta_n(|x|) \operatorname{sgn}(x)|=\theta_n(|x|)\leq n$ for $x \in \mathbb{R}$. Additionally, $\rho_n^\varepsilon(\cdot)$ is Lipschitz. Indeed, for any $x,y \in \mathbb{R}\setminus\{0\}$, we have
\begin{align} \label{eq.7}
\rho_n^\varepsilon(x)-\rho_n^\varepsilon(y) &= \theta_n(|x|) \operatorname{sgn}(x)-\theta_n(|y|) \operatorname{sgn}(y) \nonumber \\
& =(x-y) \frac{\theta_n(|x|)}{|x|}+y\left(\frac{\theta_n(|x|)}{|x|}-\frac{\theta_n(|y|) }{|y|}\right).
\end{align}
\noindent Let $\varphi_n(r):=\frac{\theta_n(|r|)}{|r|}$, $r \neq 0$. Then $\varphi_n(\cdot)$ is smooth on $(0,\infty)$ with a bounded derivative $\|\varphi_n'\|_\infty < \infty$. Thus,
using the mean value theorem, it reads
\begin{equation} \label{eq.8}
|\rho_n^\varepsilon(x)-\rho_n^\varepsilon(y)| \leq \|\varphi_n\|_\infty  \,\, |x-y| + \|\varphi'_n\|_\infty \,\, ||x|-|y|| \leq L_n |x - y|,
\end{equation}
with the Lipschitz constant $L_n=(1+n\|\varphi'_n\|_\infty)$.\\
Next, we define the truncated drift and diffusion terms to replace $\kappa (\theta-x)$ and $\xi x^{p(x)}$ of \eqref{eq.2}, respectively, as follows
\begin{equation}\label{eq.9}
f_n(x) = \kappa (\theta-\rho_n^\varepsilon(x)),
\end{equation}
and
\begin{equation}\label{eq.10}
g_n(x) = \xi (\rho_n^\varepsilon(x))^{p(\rho_n^\varepsilon(x))},
\end{equation}
yielding the following truncated form of \eqref{eq.2}
\begin{equation} \label{eq.11}
v_{n}(t) = v_0 + \int_0^t f_n(v_{n}(s)) ds + \int_0^t g_n(v_{n}(s)) dW(s),  \quad v_0>0,\,  t \in [0,T].
\end{equation}
Then clearly $f_n(\cdot)$ and $g_n(\cdot)$ satisfy the linear growth condition $(ii)$ of Lemma \ref{Lem:2.1}.\\
We proceed with showing that $f_n(\cdot)$ and $g_n(\cdot)$ are Lipschitz.\\
For $x,y \in \mathbb{R}$, we have
\begin{align}\label{eq.12}
|f_n(x)-f_n(y)| &= |\kappa (\theta-\rho_n^\varepsilon(x))-\kappa (\theta - \rho_n^\varepsilon(y))|\leq \kappa\,\,|\rho_n^\varepsilon(x)-\rho_n^\varepsilon(y)| \nonumber \\
& \leq  L^f_n\,\,|x-y|, \quad  L^f_n= \kappa L_n.
\end{align}
First we show that $z \mapsto z^{p(z)}$ is Lipschitz on $[1/n, n]$.\\
For $x \neq 0$, $|\rho_n^{\varepsilon}(x)| \in [1/n, n]$. Thus, on the compact interval $[1/n, n]$, the function $z \mapsto z^{p(z)}$ is continuously differentiable with a bounded derivative. Indeed, considering the assumption $(\mathbf{p_1})$-$(\mathbf{p_2})$, we obtain
\begin{align}\label{eq.12b}
\frac{dz^{p(z)}}{dz} & = z^{p(z)-1}p(z) + z^{p(z)} p'(z) \log z \leq n^{p^+} \left(np^+ +\|p'\|_\infty \log n\right)\leq C_n .
\end{align}
Thus, using the mean value theorem, it reads
\begin{align} \label{eq.13}
|g_n(x)-g_n(y)| &= \xi\,\,|(\rho_n^\varepsilon(x))^{p(\rho_n^\varepsilon(x))}-(\rho_n^\varepsilon(y))^{p(\rho_n^\varepsilon(y))}| \leq C_n \xi \,\, |\rho_n^\varepsilon(x)-\rho_n^\varepsilon(y)| \nonumber \\
& \leq L^g_n\,\, |x-y|, \quad  L^g_n= \xi L_n C_n.
\end{align}
Finally, letting $\hat{L}_n = \max\{ (L^f_n)^2, (L^g_n)^2\}$ makes $f_n(\cdot)$ and $g_n(\cdot)$ Lipschitz with constant $\hat{L}_n$; that is, for $x,y \in \mathbb{R}$,
\begin{align}\label{eq.14}
|f_n(x)-f_n(y)|^2 \vee |g_n(x)-g_n(y)|^2 & \leq \hat{L}_n\,|x-y|^2.
\end{align}
Since $f_n(\cdot)$ and $g_n(\cdot)$ satisfy both the Lipschitz  and the linear growth conditions, we can apply Lemma \ref{Lem:2.1}; that is, for each integer $n \geq 1$, there exists a unique solution $v_n(t)$ in $\mathcal{M}^{2}[0,T]$ to the truncated equation \eqref{eq.11}.\\
Notice that using the H\"{o}lder inequality and It\^{o}'s isometry show that $f_n(\cdot)$ and $g_n(\cdot)$ belong to $\mathcal{M}^{2}[0,T]$ since
\begin{equation}\label{eq.17a}
\mathbb{E}\bigg|\int_{0}^{t}f_n(v_{n}(s))ds\bigg|^{2} \leq t\mathbb{E}\int_{0}^{t}|f_n(v_{n}(s))|^{2}ds\leq T K \mathbb{E}\int_{0}^{t}(1+|v_{n}(s)|^{2})ds <\infty,
\end{equation}
and
\begin{equation}\label{eq.17b}
\mathbb{E}\bigg|\int_{0}^{t}g_n(v_{n}(s))dW(s)\bigg|^{2} \leq \mathbb{E}\int_{0}^{t}|g_n(v_{n}(s))|^{2}ds\leq K \mathbb{E}\int_{0}^{t} (1+|v_{n}(s)|^{2})ds <\infty,
\end{equation}
for, $v_n(t)$ in $\mathcal{M}^{2}[0,T]$, $n\geq 1$ and $t \in [0,T]$.\\
In the sequel, to construct a solution to the original equation \eqref{eq.2}, we  mimic the deterministic situation and construct recursively a sequence of successive approximations (the Picard iterations see, e.g., \cite{friedman75,karatzasshreve91,mao07}) by setting
\begin{equation} \label{eq.17c}
v^{(0)}_n(t)=v_0,
\end{equation}
and
\begin{equation} \label{eq.18}
v^{(k+1)}_{n}(t) = v_0 + \int_0^t f_n(v^{(k)}_{n}(s)) ds + \int_0^t g_n(v^{(k)}_{n}(s)) dW(s),\,\, k\geq 0,\,\,  t \in [0,T].
\end{equation}
Then, we show that the sequence $\{v^{(k)}_{n}(t)\}_{k \geq 0}$ will converge to a solution of the original equation \eqref{eq.2}.\\
First, we start with the uniform $2$nd moment estimate of $v_n(t)$ in $k\geq 0$ to make sure that each iterate remains bounded
in $L^2$-norm; that is, we show that
\begin{equation}\label{eq.19}
\sup_{k\geq 0}\left(\mathbb{E}\sup_{0 \leq t \leq T}|v^{(k)}_{n}(t)|^{2} \right)\leq (1+ 3\mathbb{E}|v_0|^2)\exp(3 K((T^2+4T)).
\end{equation}
\noindent Fix a truncation level $n\geq1$. Applying the H\"{o}lder inequality, and the Burkholder-Davis-Gundy (BDG) inequality together yields
\begin{align} \label{eq.21}
\mathbb{E}\sup_{0\leq s\leq t}|v^{(k+1)}_{n}(s)|^{2}& \leq  3 \mathbb{E}|v_0|^2 + 3\mathbb{E}\bigg|\int_0^t f_n(v^{(k)}_{n}(s)) ds\bigg|^2 \nonumber \\
&+ 3\mathbb{E}\sup_{0 \leq s \leq t}\bigg|\int_0^s g_n(v^{(k)}_{n}(u)) dW(u)\bigg|^2\nonumber \\
&\leq 3 \mathbb{E}|v_0|^2 + 3TK\int_0^t (1+\mathbb{E}|v^{(k)}_{n}(u)|^2) du \nonumber \\
& + 12 \mathbb{E} \int_0^t |g_n(v^{(k)}_{n}(u))|^2 dW(u)\nonumber \\
& \leq 3 \mathbb{E}|v_0|^2 + 3 K(T+4) \int_0^t (1+\mathbb{E}|v^{(k)}_{n}(u)|^2) du.
\end{align}
Consequently,
\begin{align} \label{eq.22}
1+\mathbb{E}\sup_{0\leq s\leq t}|v^{(k+1)}_{n}(s)|^{2} & \leq 1+ 3\mathbb{E}|v_0|^2\nonumber \\
& + 3 K(T+4) \int_0^t \left(1+\mathbb{E}\sup_{0\leq w\leq u}|v^{(k)}_{n}(w)|^{2} \right) du.
\end{align}
Set
\begin{equation} \label{eq.23}
M_{k}(t)=\mathbb{E}\sup_{0 \leq s \leq t}|v^{(k)}_{n}(s)|^{2}, \quad  t \in [0,T].
\end{equation}
Given a family of real‐valued functions $\{M_{k}(t)\}_{k\geq 0}$ on $[0,T]$, define the function $A(t):[0,T] \to \mathbb{R}$ by
\begin{equation} \label{eq.24}
A(t)=\sup_{k\geq 0}M_{k}(t).
\end{equation}
Thus,
\begin{align} \label{eq.25}
1+M_{k+1}(t)& \leq 1+ 3\mathbb{E}|v_0|^2 + 3 K(T+4) \int_0^t \left(1+A(u)\right) du, \quad \forall k.
\end{align}
Taking the supremum over $k$ gives
\begin{align} \label{eq.26}
1+A(t)& \leq 1+ 3\mathbb{E}|v_0|^2 + 3 K(T+4) \int_0^t \left(1+A(u)\right) du, \quad \forall k.
\end{align}
Now, applying the Gronwall inequality yields
\begin{align} \label{eq.27}
A(t)& \leq (1+ 3\mathbb{E}|v_0|^2)\exp(3 K((T^2+4T)).
\end{align}
Finally, using \eqref{eq.23} and \eqref{eq.24} provides the uniform bound in $k$
\begin{align*}
\sup_{k\geq 0}\left(\mathbb{E}\sup_{0 \leq t \leq T}|v^{(k)}_{n}(t)|^{2} \right)\leq (1+ 3\mathbb{E}|v_0|^2)\exp(3 K((T^2+4T)),
\end{align*}
which is the desired result.\\
Now, we make the inductive assumption that $v^{(k)}_{n}(t) \in \mathcal{M}^{2}[0,T]$ and it holds
\begin{align} \label{eq.18a}
\mathbb{E}|v^{(m+1)}_{n}(t)-v^{(m)}_{n}(t)|^2 & \leq \frac{(\hat{M}t)^{m+1}}{(m+1)!},\quad 0\leq m \leq k-1,\,\,  t \in [0,T].
\end{align}
where $\hat{M}>0$ is a constant depending on $K,\hat{L}_n$ and $T$.\\
Notice first that since $v_0$ is $\mathcal{F}_0$-measurable, $v^{(k+1)}_{n}(t)$ is well-defined if $k=0$. Then
\begin{align} \label{eq.29}
|v^{(1)}_{n}(t)-v^{(0)}_{n}(t)|^2=|v^{(1)}_{n}(t)-v_{0}|^2 & \leq 2\bigg|\int_0^t f_n(v^{(0)}_{n}(s)) ds\bigg|^2
+2\bigg|\int_0^t g_n(v^{(0)}_{n}(s))dW(s)\bigg|^2.
\end{align}
Taking expectation and applying the H\"{o}lder inequality, and It\^{o}'s isometry it reads
\begin{align} \label{eq.30}
\mathbb{E}|v^{(1)}_{n}(t)-v_{0}|^2 & \leq 2 t \mathbb{E}\int_0^t |f_n(v^{(0)}_{n}(s))|^2 ds +2 \mathbb{E}\int_0^t |g_n(v^{(0)}_{n}(s))|^2 ds \nonumber \\
& \leq 2t K \mathbb{E}\int_0^t (1+|v_0|^{2}) ds +2 K \mathbb{E}\int_0^t (1+|v_0|^{2}) ds \nonumber \\
& \leq 2t K \int_0^t (1+\mathbb{E}|v_0|^{2}) ds +2 K \int_0^t (1+\mathbb{E}|v_0|^{2}) ds \nonumber \\
& \leq 2t^2 K  (1+\mathbb{E}|v_0|^{2})+2t K (1+\mathbb{E}|v_0|^{2}) \nonumber \\
& \leq M_1t, \nonumber \\
\end{align}
which implies that $v^{(1)}_{n}(t) \in \mathcal{M}^{2}[0,T]$, where $M_1=2 K (T+1)(1+\mathbb{E}|v_0|^{2})$. From \eqref{eq.30}, it is easy to see that \eqref{eq.18a} holds for $k=0$. Now, assume that \eqref{eq.18a} holds for some $k\geq 0$. Then taking expectation and using \eqref{eq.14}, it reads
\begin{align} \label{eq.31}
\mathbb{E}|v^{(k+1)}_{n}(t)-v^{(k)}_{n}(t)|^2 & \leq 2\bigg|\int_0^t (f_n(v^{(k)}_{n}(s))-f_n(v^{(k-1)}_{n}(s))) ds\bigg|^2 \nonumber \\
& +2\bigg|\int_0^t (g_n(v^{(k)}_{n}(s))-g_n(v^{(k-1)}_{n}(s)))dW(s)\bigg|^2 \nonumber \\
& \leq 2 t\hat{L}_n \mathbb{E}\int_0^t |v^{(k)}_{n}(s)-v^{(k-1)}_{n}(s)|^2 ds\nonumber \\
& +2\hat{L}_n \mathbb{E}\int_0^t |v^{(k)}_{n}(s)-v^{(k-1)}_{n}(s)|^2 ds,
\end{align}
which gives
\begin{align} \label{eq.32}
\mathbb{E}|v^{(k+1)}_{n}(t)-v^{(k)}_{n}(t)|^2 & \leq M_2  \int_0^t \mathbb{E} |v^{(k)}_{n}(s)-v^{(k-1)}_{n}(s)|^2 ds,
\end{align}
where $M_2= 2\hat{L}_n(T+1)$. If we let $m=k-1$ in \eqref{eq.18a} and substitute into the right-hand side of \eqref{eq.32}, it yields
\begin{align} \label{eq.33}
\mathbb{E}|v^{(k+1)}_{n}(t)-v^{(k)}_{n}(t)|^2 & \leq \hat{M} \int_0^t \frac{(\hat{M}s)^{k}}{(k)!}ds=\frac{(\hat{M}t)^{k+1}}{(k+1)!},\quad \hat{M}=\max\{M_1,M_2\}.
\end{align}
Therefore \eqref{eq.18a} holds for $m=k$, which means that $v^{(k+1)}_{n}(t) \in \mathcal{M}^{2}[0,T]$.\\
Additionally,
\begin{align} \label{eq.34}
\sup_{0\leq t \leq T}|v^{(k+1)}_{n}(t)-v^{(k)}_{n}(t)|^2 & \leq 2 T \hat{L}_n \int_0^T |v^{(k)}_{n}(s)-v^{(k-1)}_{n}(s)|^2 ds\nonumber \\
& +2\sup_{0\leq t \leq T}\bigg|\int_0^T (g_n(v^{(k)}_{n}(s))-g_n(v^{(k-1)}_{n}(s)))dW(s)\bigg|^2.
\end{align}
\noindent Then taking expectation and using the BDG inequality along with \eqref{eq.18a}, it reads
\begin{align} \label{eq.35}
\mathbb{E}\sup_{0\leq t \leq T}|v^{(k+1)}_{n}(t)-v^{(k)}_{n}(t)|^2 & \leq 2 T \hat{L}_n \int_0^T \mathbb{E}|v^{(k)}_{n}(s)-v^{(k-1)}_{n}(s)|^2 ds\nonumber \\
& +8\hat{L}_n\int_0^T \mathbb{E}|v^{(k)}_{n}(s)-v^{(k-1)}_{n}(s)|^2 ds \nonumber \\
& \leq 2 T \hat{L}_n \int_0^T \frac{(\hat{M}s)^{k}}{(k)!}ds + 8\hat{L}_n\int_0^T \frac{(\hat{M}s)^{k}}{(k)!}ds\nonumber \\
&\leq M_3 \frac{(\hat{M}T)^{k}}{(k)!},
\end{align}
where $M_3= 2\hat{L}_nT(T+4)$. Then, using the Markov inequality it reads
\begin{align} \label{eq.36}
\mathbb{P}\left\{\omega;\, \sup_{0\leq t \leq T}|v^{(k+1)}_{n}(t)-v^{(k)}_{n}(t)|\geq \frac{1}{2^k} \right\}&\leq 2^{2k}\mathbb{E}\sup_{0\leq t \leq T}|v^{(k+1)}_{n}(t)-v^{(k)}_{n}(t)|^2 \nonumber \\
&\leq  M_3 \frac{(4\hat{M}T)^{k}}{(k)!}.
\end{align}
Since $\sum^{\infty}_{k=0}M_3 \frac{(4\hat{M}T)^{k}}{(k)!}<\infty$, using the Borel-Cantelli lemma we obtain
\begin{align} \label{eq.37}
\mathbb{P}\left\{\limsup_{k \to \infty}\left\{\omega;\,\sup_{0\leq t \leq T}|v^{(k+1)}_{n}(t)-v^{(k)}_{n}(t)|\geq \frac{1}{2^k}\right\} \right\}&=0.
\end{align}
Therefore, for almost every (a.e.) $\omega \in \Omega$ there exists a positive integer $k_0(\omega)$ such that
\begin{align} \label{eq.38}
\sup_{0\leq t \leq T}|v^{(k+1)}_{n}(t)-v^{(k)}_{n}(t)|\leq \frac{1}{2^k}\,\,\text{ if }\,\, k\geq k_0(\omega).
\end{align}
Thus, the partial sums
\begin{align} \label{eq.39}
v_0+\sum_{m=0}^{k-1}(v^{(m+1)}_{n}(t)-v^{(m)}_{n}(t))=v^{(k)}_{n}(t)
\end{align}
are convergent uniformly in $t \in [0,T]$. Let's denote this limit by $v_{n}(t)$. Then $v_{n}(t)$ is a continuous and $\mathcal{F}_t$-adapted process. Furthermore, by \eqref{eq.18a}, $\{v^{(k)}_{n}(t)\}_{k \geq 0}$, $t \geq 0$, is a Cauchy sequence in $L^2([0, T] \times \Omega)$, and hence, we have the convergence $v^{(k)}_{n}(t) \to v_{n}(t)$ in $L^2([0, T] \times \Omega)$ as $k \to \infty$.
Further, using Fatou's lemma in \eqref{eq.19} for $k \to \infty$ yields
\begin{equation}\label{eq.40}
\mathbb{E}|v_{n}(t)|^{2}\leq (1+ 3\mathbb{E}|v_0|^2)\exp(3 K((T^2+4T)),
\end{equation}
which shows that $v_{n}(t) \in \mathcal{M}^{2}[0,T]$. Considering that $f_n(\cdot)$, $g_n(\cdot)$ are Lipschitz and uniformly bounded in $k$, and applying the H\"{o}lder inequality and It\^{o}'s isometry, it yields
\begin{align} \label{eq.41}
\int_0^t f_n(v^{(k)}_{n}(s))ds \xrightarrow{L^2} \int_0^t f_n(v_{n}(s))ds,
\end{align}
and
\begin{align} \label{eq.42}
\int_0^t g_n(v^{(k)}_{n}(s))dW(s) \xrightarrow{L^2} \int_0^t g_n(v_{n}(s))dW(s),
\end{align}
as  $k \to \infty$. Hence, if we take $k \to \infty$ in \eqref{eq.18} we conclude that
\begin{equation} \label{eq.43}
v_{n}(t) = v_0 + \int_0^t f_n(v_{n}(s)) ds + \int_0^t g_n(v_{n}(s)) dW(s),\,\,  t \in [0,T].
\end{equation}
Thus, $v_{n}(t)$ is a solution to \eqref{eq.11}. Next, we shall show the uniqueness. Let assume that $v_{n}(t)$ and $\hat{v}_{n}(t)$ be to solutions of \eqref{eq.11} in $\mathcal{M}^{2}[0,T]$. Let us define the stopping times
\begin{equation} \label{eq.44}
\tau_n=\inf\{t\geq0;\, v_{n}(t) \notin [1/n,n]\},
\end{equation}
and
\begin{equation} \label{eq.45}
\hat{\tau}_n=\inf\{t\geq0;\, \hat{v}_{n}(t) \notin [1/n,n]\},
\end{equation}
and set $s_n=\tau_n \cap \hat{\tau}_n$. Then $\lim_{n \to \infty}s_n=+\infty$ a.s. since the solutions will not explode in finite time. Then
\begin{align} \label{eq.46}
v_{n}(t\wedge s_n) - \hat{v}_{n}(t\wedge s_n) &= \int_0^{t\wedge s_n} (f_n(v_{n}(u))-f_n(\hat{v}_{n}(u))) du \nonumber \\
& + \int_0^{t\wedge s_n} (g_n(v_{n}(u))-g_n(\hat{v}_{n}(u))) dW(u).
\end{align}
Then taking expectation, applying the H\"{o}lder inequality, It\^{o}'s isometry and using \eqref{eq.14}, it reads
\begin{align} \label{eq.47}
\mathbb{E}|v_{n}(t\wedge s_n) - \hat{v}_{n}(t\wedge s_n)|^2 & \leq 2\bigg|\int_0^{t\wedge s_n} (f_n(v_{n}(u))-f_n(\hat{v}_{n}(u))) du\bigg|^2 \nonumber \\
& +2\bigg|\int_0^{t\wedge s_n} (g_n(v_{n}(u))-g_n(\hat{v}_{n}(u))) dW(u)\bigg|^2 \nonumber \\
& \leq 2 t\hat{L}_n \mathbb{E}\int_0^{t\wedge s_n} |v_{n}(u)-\hat{v}_{n}(u)|^2 du\nonumber \\
& +2\hat{L}_n \mathbb{E}\int_0^{t\wedge s_n} |v_{n}(u)-\hat{v}_{n}(u)|^2 du \nonumber \\
& \leq 2\hat{L}_n(T+1) \int_0^{t} \mathbb{E}|v_{n}(u\wedge s_n)-\hat{v}_{n}(u\wedge s_n)|^2 du.
\end{align}
Now, applying the Gronwall inequality yields $\mathbb{E}|v_{n}(t\wedge s_n)-\hat{v}_{n}(t\wedge s_n)|^2=0$, and hence, $v_{n}(t\wedge s_n)=\hat{v}_{n}(t\wedge s_n)$ a.s. for each $t \geq 0$. Since $v_{n}(t)$ and $\hat{v}_{n}(t)$ are solutions of \eqref{eq.11}, they are continuous. Thus, the processes $\{v_{n}(t\wedge s_n);\, 0\leq t <\infty \}$ and $\{\hat{v}_{n}(t\wedge s_n);\, 0\leq t <\infty \}$ are indistinguishable, i.e. their paths coincide a.s. if they start from the same initial point, $v_0$. Further, since $\lim_{n \to \infty}s_n=+\infty$ a.s. for any fixed $t$, $t\wedge s_n \to t$ as $n \to \infty$. Therefore, $v_{n}(t\wedge s_n)\to v_{n}(t)$ and $\hat{v}_{n}(t\wedge s_n) \to \hat{v}_{n}(t)$, and the indistinguishability extends to $\{v_{n}(t);\, 0\leq t <\infty \}$ and $\{\hat{v}_{n}(t);\, 0\leq t <\infty \}$.\\

\noindent Lastly, we  show that the solution $v_{n}(t)$ of the truncated equation  \eqref{eq.11} is defined for all $t \in [0,T]$; that is, we construct a global solution to \eqref{eq.1}.\\
For the equation \eqref{eq.11} define the stopping time $\tau_n$ by
\begin{equation} \label{eq.48}
\tau_n=T \wedge \inf\{t\geq0;\, v_{n}(t) \notin [1/n,n]\}.
\end{equation}
Note that $\tau_n \uparrow T$ and the solutions $v_{n}(t)$ and $v_{n+1}(t)$ agree up to the stopping time, that is
\begin{equation} \label{eq.49}
v_{n}(t)= v_{n+1}(t)\quad \text{for } 0\leq t \leq \tau_n,
\end{equation}
which means that $\tau_{n+1}\geq \tau_n$ when $n+1\geq n$.\\
\textbf{Case I}: If $0\leq v_n(t)< 1/n$, $t \in [0,T]$; then the truncated equation \eqref{eq.11} becomes
\begin{equation} \label{eq.50}
dv_{n}(t) = \kappa\left(\theta-\frac{1}{n}\right) dt + \xi \left(\frac{1}{n} \right)^{p(1/n)} dW(t), \quad t \in [0,T].
\end{equation}
Now, using the Feller's test for the equation \eqref{eq.50} yields
\begin{align}\label{eq.51}
\lim_{x \to 0^+}\mathcal{T}_{p}(x)& = \kappa\left(\theta-\frac{1}{n}\right)\geq 0.
\end{align}
Note that \eqref{eq.51} holds as long as $\theta \geq \frac{1}{n}$; however, since we're interested only in sufficiently large $n$, this is consistent with the actual meaning of $\theta$. Therefore, the solution $v_n(t)$ of \eqref{eq.50} a.s. never reaches zero since $v_0>0$. This means, for a.e. sample path $\omega \in \Omega,\, v_n(t,\omega)>0$ for all $t \in [0,T]$. Further, since $v_n(\cdot,\omega)$ is a continuous function on the compact interval $[0,T]$, for a.e. $\omega \in \Omega$ it assumes its minimum on this interval. Thus, for a.e. $\omega \in \Omega$, there is a constant $\eta(\omega)>0$ such that
\begin{align}\label{eq.52}
\inf_{0\leq t \leq T} v_n(t,\omega)=\eta(\omega).
\end{align}
Define the set
\begin{align}\label{eq.53}
B_n=\left\{\omega;\, \inf_{0\leq t \leq T} v_n(t,\omega)< \frac{1}{n}\right\}.
\end{align}
For a.e. $\omega \in \Omega$ choose $n_0(\omega) \in \mathbb{N}$ such that $n_0(\omega) \geq \max\left\{1, \frac{1}{\eta(\omega)} \right\}$. Then for $n \geq n_0(\omega)$, we have
\begin{align}\label{eq.54}
\inf_{0\leq t \leq T} v_n(t,\omega)\geq \eta(\omega) \geq \frac{1}{n},
\end{align}
and hence, $\omega \notin B_n$. Thus, $\Omega \backslash B_n=B^{c}_n=\left\{\omega;\, \inf_{0\leq t \leq T} v_n(t,\omega) \geq \frac{1}{n}\right\} \neq \emptyset$ for sufficiently large $n$. Then clearly
\begin{align} \label{eq.55}
\mathbb{P}\left\{\bigcup^{\infty}_{k=1} \bigcap^{\infty}_{n=k} B^{c}_n \right\}&=1,
\end{align}
which yields
\begin{align} \label{eq.56}
\mathbb{P}\left\{\liminf_{n \to \infty} B_n \right\}&=0.
\end{align}
\noindent \textbf{Case II}: If $v_n(t)>n$, $t \in [0,T]$; then using the Markov inequality and Lemma \ref{Lem:2.1} it reads
\begin{align} \label{eq.57}
\mathbb{P}\left\{\omega;\, \sup_{0\leq t \leq T}|v_{n+1}(t)-v_{n}(t)|>0 \right\}&=\mathbb{P}\left\{\omega;\, \sup_{0\leq t \leq T}|v_{n}(t)|> n \right\}\nonumber \\
&\leq \frac{1}{n^2}\,\mathbb{E}\sup_{0\leq t \leq T}|v_{n}(t)|^2 \nonumber \\
&\leq \frac{1}{n^2}(1+ 3\mathbb{E}|v_0|^2)\exp(3 KT(T+4))<\infty.
\end{align}
Set $D_n=\left\{\omega;\, \sup_{0\leq t \leq T} |v_n(t,\omega)|> n \right\}$. Then using the Borel-Cantelli lemma we obtain
\begin{align} \label{eq.58}
\mathbb{P}\left\{\limsup_{n \to \infty}D_n \right\}&=0.
\end{align}
Next, define
\begin{align} \label{eq.59}
\lim_{n \to \infty} v_n(t,\omega)=v(t,\omega)\,\, \text{ if }\, \omega \notin \lim_{n \to \infty} (B_n \cup D_n).
\end{align}
Then $v_n(t,\omega)=v(t,\omega)$ for all $t \in [0,T]$ if $\omega \notin B_n \cup D_n$. \\
Now, applying all the information gathered so far to the truncated equation  \eqref{eq.11}, we obtain
\begin{equation} \label{eq.60}
v(t) = v_0 + \int_0^t f(v(s)) ds + \int_0^t g(v(s)) dW(s)\,\, \text{a.s. if }\omega \notin B_n \cup D_n,\,\, t \in [0,T],
\end{equation}
that is, $v(t)$ is the unique solution of \eqref{eq.1} defined at all times  $t \in [0,T]$.
\end{proof}

\section{Moment estimation of the solution}\label{Sec4}

In this subsection, by assuming that $v(t)$, $t\in[0,T]$ with the initial value $v(0)=v_0>0$ is the unique solution for the equation \eqref{eq.1}, we obtain an estimate on the $m$th moment of the solution.
\begin{theorem}\label{Thrm:4.1}
Let $m\geq 2$ be some positive integer and $\mathbb{E}v_0^{m}<\infty$. Then
\begin{align} \label{eq.61}
\mathbb{E}v(t)^{m} &\leq 2^{m-1} (1+\mathbb{E}v_0^{m})\exp(C_mt),
\end{align}
where $C_m=m\kappa(\theta+1)+\frac{\xi^2}{2} m(m-1)$.
\end{theorem}
\begin{proof}
We apply It\^{o}'s formula along with the stopping time argument. Then
\begin{align} \label{eq.62a}
dv(t)^{m} &= m\kappa\theta v(t)^{m-1}dt- m\kappa v(t)^{m}dt+ m\xi v(t)^{m-1+p(v(t))}dW(t)  \nonumber \\
&+\frac{\xi^2}{2} m(m-1) v(t)^{m-2+2p(v(t))}dt.
\end{align}
Integrating over $t\in[0,T]$, and considering the assumptions on $p(\cdot)$ it reads
\begin{align} \label{eq.62b}
v(t)^{m} & \leq (1+v_0)^{m} + m\kappa\theta \int_{0}^{t} v(s)^{m-1}ds+ m\kappa \int_{0}^{t} v(s)^{m}ds  \nonumber \\
&+ m\xi \int_{0}^{t} v(s)^{m-1+p(v(s))}dW(s)+ \frac{\xi^2}{2} m(m-1)\int_{0}^{t} v(s)^{m-2+2p(v(s))}ds \nonumber \\
& \leq 2^{m-1}(1+v_0^{m}) + \left(m\kappa(\theta+1)+\frac{\xi^2}{2} m(m-1) \right)\int_{0}^{t} v(s)^{m}ds  \nonumber \\
& + m\xi \int_{0}^{t} v(s)^{m-1+p(v(s))}dW(s).
\end{align}
For every integer $n\geq 1$, let $\tau_n$ be the stopping time as defined in \eqref{eq.48}. Since
\begin{align} \label{eq.63}
v(s)^{2(m-1+p(v(s)))}&\leq v(s)^{2(m-1+p^+)}\leq v(s)^{2(m+p^+)} < n^{2(m+p^+)}<\infty,
\end{align}
for $0\leq s \leq t \wedge \tau_n$, we have
\begin{align} \label{eq.64}
\int_{0}^{t \wedge \tau_n}\mathbb{E}v(s)^{2(m-1+p(v(s)))}ds < \infty.
\end{align}
Thus, taking expectations and using the zero mean property of the It\^{o} integral yields
\begin{align} \label{eq.65}
\mathbb{E}v(t \wedge \tau_n)^{m} & \leq 2^{m-1}(1+\mathbb{E}v_0^{m}) + C_m \int_{0}^{t \wedge \tau_n} \mathbb{E}v(s)^{m}ds  \nonumber \\
& \leq 2^{m-1}(1+\mathbb{E}v_0^{m}) + C_m \int_{0}^{t} \mathbb{E}v(s \wedge \tau_n)^{m}ds.
\end{align}
Applying the Gronwall inequality gives
\begin{align} \label{eq.66}
\mathbb{E}v(t \wedge \tau_n)^{m} & \leq 2^{m-1}(1+\mathbb{E}v_0^{m})\exp(C_mt).
\end{align}
Lastly, letting $n \to \infty$, and considering that $\tau_n \to T$ as $n \to \infty$ a.s., concludes
\begin{align} \label{eq.67}
\mathbb{E}v(t)^{m} & \leq 2^{m-1}(1+\mathbb{E}v_0^{m})\exp(C_mt).
\end{align}
\end{proof}

\section{Numerical implementation}\label{Sec5}
In this section, we test our model, GM, with three alternative variable exponent functions $p(\cdot)$. Each variable exponent function is paired with the classical CIR model in sample‐path and distributional comparisons. The findings, visualized in Figures \ref{fig1:cir_p1}, \ref{fig2:cir_p2} and \ref{fig3:cir_p3}, support the theoretical validity of the GM model as a meaningful generalization of the CIR model.
\subsection{Experiment design}
\noindent We compare the solution of the CIR model
against
\begin{equation*}
\mathrm{(GM)}_i
\quad dv(t)=\kappa(\theta - v(t))dt + \xi v(t)^{p_i(v(t))}dW(t),
\quad i=1,2,3,
\end{equation*}
with
\begin{equation*}
p_1(v) = 0.5 + 0.3(1 - e^{-v}), \quad p_2(v) = 0.6 + 0.2 \tanh(v), \quad \text{and} \quad p_3(v) = 0.55 + 0.2 \frac{v}{1 + v}.
\end{equation*}
Parameters are set as $\kappa = 2.0$, $\theta = 0.05$, $\xi = 0.3$, $v(0) = 0.05$, $T = 1.0$, $\Delta t = 10^{-3}$, and $N = 5000$ sample paths. Random draws use "np.random.seed(42)" for the Monte Carlo simulatons and "np.random.seed(0)" for the single paired sample path. We enforce $v_{k+1}=\max\{\,v_k+\kappa(\theta-v_k)\Delta t+\xi\,v_k^{p(v_k)}\Delta W_k,\;10^{-6}\}$, where the lower bound $10^{-6}$ is the positivity truncation, sharing $\Delta W_k\sim N(0,\Delta t)$ between CIR and each generalized model, $(GM)_{i}$.
The experiments use identical Brownian increments to isolate the impact of the variable exponent $p(\cdot)$. The Euler-Maruyama method is used to discretize the stochastic differential equations. To obtain the distribution of the process, Monte Carlo simulations are performed.

\subsection{Analysis}
\noindent All three variable exponent functions preserve the positive‐mean‐reverting behavior of the CIR process. The choice of $p(\cdot)$ significantly affects both sample path behavior and terminal distributions. The rate of change of $p(\cdot)$ influences the volatility clustering behavior. Empirical moments remain below the theoretical bounds obtained by Theorem \ref{Thrm:4.1}. Our simulations with different types of $p(\cdot)$ has shown that GM provides significant control over volatility dynamics while maintaining theoretical soundness since each exponent function offers distinct advantages for modeling specific market regimes.For the variable exponent functions used in this paper, for example:\\
$p_1$ is flexible due to its smooth transition, but its sensitivity to large $v$ values may lead to instability. Thus, it is suitable for markets with moderate regime shifts, where volatility increases gradually with variance.\\
$p_2$ captures strong regime dependence but its rapid growth for large $v$ may exaggerate volatility. Therefore, it might be ideal for markets with abrupt regime changes, such as during financial crises.\\
$p_3$ offers stability and gradual adjustment; however, its limited range may under represent extreme volatility, and hence, it could be appropriate for stable markets with mild regime-dependent dynamics, such as steady economic growth periods.\\
As a next step, one may extend these experiments by calibrating the parameters $(\kappa,\theta,\xi)$ to market data, comparing implied‐volatility surfaces, or studying long‐time behavior under each GM$_i$, either by using the variable exponent functions $p(\cdot)$ provided in this paper or considering different ones.

\begin{figure}[tp]
  \centering
  \includegraphics[width=0.8\textwidth]{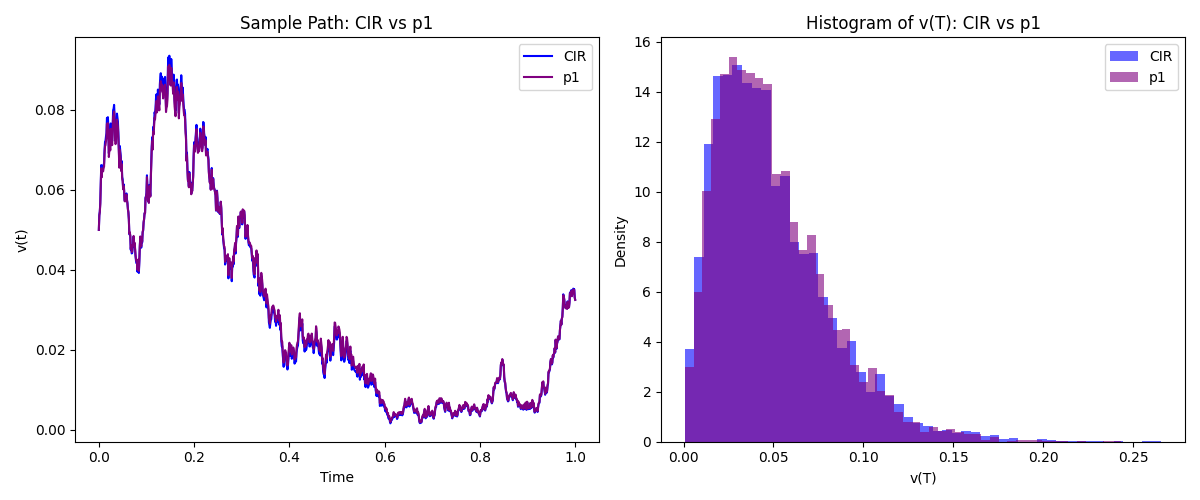}
  \caption{Sample path (left) and terminal distribution (right) for CIR and $p_1(v)$.}  \label{fig1:cir_p1}
\end{figure}

\begin{figure}[tp]
  \centering
  \includegraphics[width=0.8\textwidth]{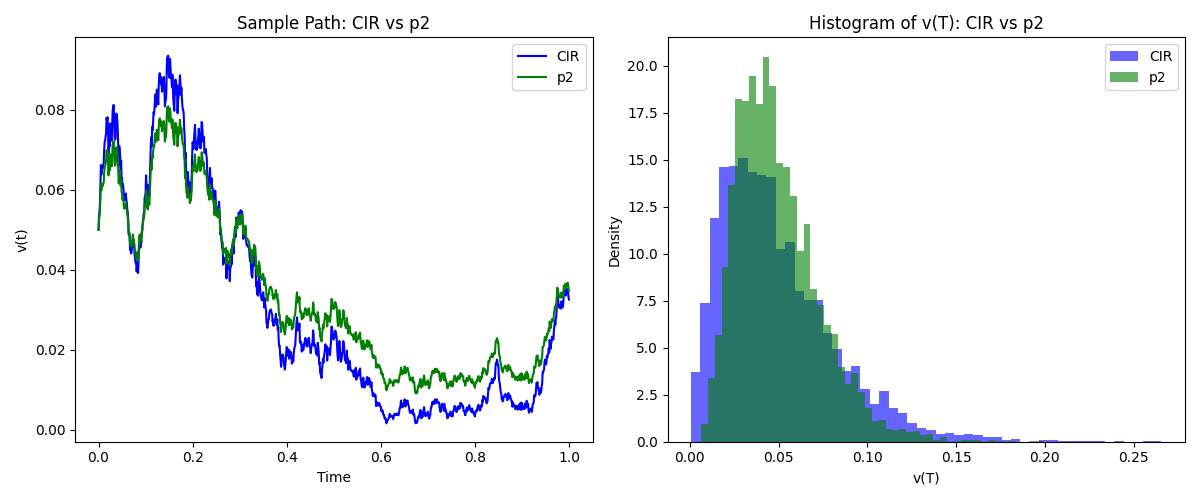}
  \caption{Sample path (left) and terminal distribution (right) for CIR and $p_2(v)$.}  \label{fig2:cir_p2}
\end{figure}

\begin{figure}[tp]
  \centering
  \includegraphics[width=0.8\textwidth]{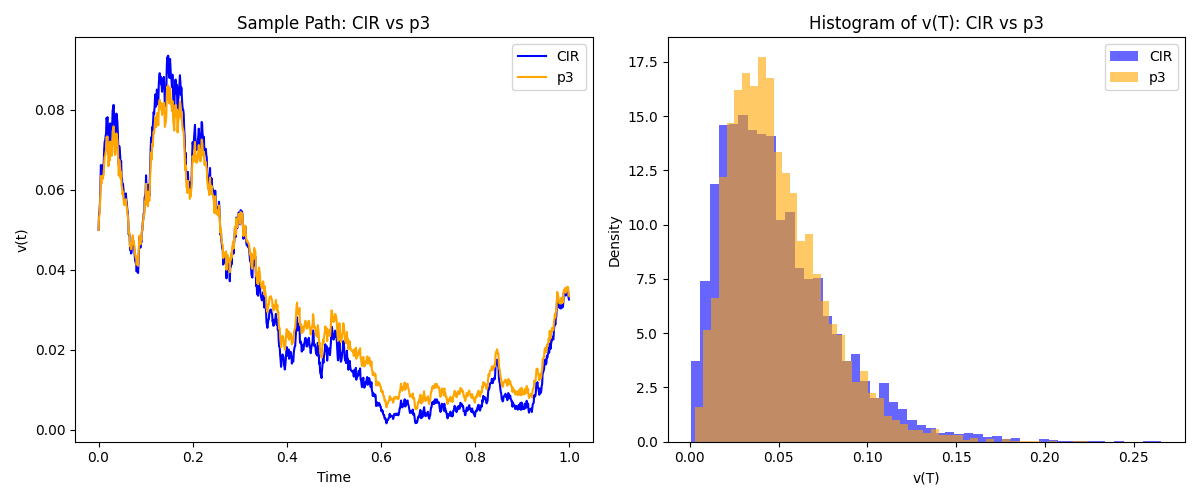}
  \caption{Sample path (left) and terminal distribution (right) for CIR and $p_3(v)$.} \label{fig3:cir_p3}
\end{figure}

\begin{remark} We aware of the fact that the explicit Euler-Maruyama method lacks the usual strong and weak convergence assurances when applied to stochastic differential equations whose coefficients fail to satisfy Lipschitz condition (see e.g., \cite{hjk11}). Fortunately, a great amount of research in numerical analysis has provided significant stabilized and implicit modifications, such as tamed Euler, semi-implicit approaches, backward Euler-Maruyama, which help restore convergence assurances and preserve qualitative characteristics across many non-Lipschitz cases. (see e.g., \cite{a10,hms02,s13}).\\
Further, we'd like to point out that since the numerical experiments in this paper solely intended to visualize the qualitative effects of replacing $\sqrt{v}$ by $v^{p(v)}$, we used the explicit EM scheme with a small time step together with a positivity truncation (lower bound $10^{-6}$) for stability in Monte Carlo comparisons. As well, the existence, uniqueness and higher-moment estimates obtained in the paper are proved for the continuous-time stochastic differential equations and do not depend on the choice of any numerical discretization. The Monte Carlo simulations are included exclusively for demonstration purposes; while they support the qualitative predictions of the theory and are complementary, they neither constitute formal proof nor serve as a replacement for the rigorous analysis provided.
\end{remark}

\section{Conclusion}\label{Sec6}
We propose a \textit{variable exponent}  generalization of the CIR model by replacing $\sqrt{v(t)}$ with $v(t)^{p(v(t))}$, where the state-dependent exponent function $p(\cdot)$ satisfies natural regularity and boundedness conditions. Existence, uniqueness, and higher‐moment estimates of the solution are obtained. The numerical experiments provide strong evidence that the generalized CIR model with variable exponent functions offers a flexible framework for volatility modeling while maintaining the essential characteristics of the classical CIR process. The choice of the specific form of $p(\cdot)$ allows for fine-tuning the model's behavior to match empirical observations or theoretical requirements. The success of these experiments validates both the theoretical analysis and the numerical implementation, confirming that the proposed generalization represents a reasonable generalization of the classical CIR framework.

\section{Appendix: It\^o vs Stratonovich interpretations of the model}\label{Sec7}
In stochastic calculus, the It\^o and Stratonovich integrals represent two distinct interpretations of SDEs. That is, we can interpret our model by using either
the  It\^o form, i.e. \eqref{eq.1}
\begin{align}\label{eq.68}
\mathrm{(GM)}
\quad
dv(t) &= \kappa(\theta - v(t))dt + \xi v(t)^{p(v(t))} dW(t),
\end{align}
or the Stratonovich form
\begin{align}\label{eq.69}
\mathrm{(GMS)}
\quad
dv(t) &= \kappa(\theta - v(t))dt + \xi v(t)^{p(v(t))}\circ dW(t).
\end{align}
While both forms are mathematically valid, they yield different results when applied to models with state-dependent diffusion coefficients. On the other hand, using the conversion formula, we can write It\^o counterpart of GMS as follows
\begin{align}\label{eq.70}
\mathrm{(GMSI)}
\quad
dv(t) &= \left(\kappa(\theta - v(t)) +  h(v(t))\right) dt + \xi v(t)^{p(v(t))} dW(t),
\end{align}
where
\begin{align}\label{eq.70a}
h(v(t))=\frac{1}{2} \xi^2 v(t)^{2p(v(t))} \left(p^{\prime}(v(t)) \log v(t) + \frac{p(v(t))}{v(t)}\right).
\end{align}
is the drift-correction term. Since GMSI is mathematically equivalent to GMS, to conduct our analysis in order to underline the differences in results, it is enough to compare GM with GMSI.\\

First, we'd like to note that GMSI is also well-defined; that is, if the process $v(t)$ satisfies GMSI, then under the assumptions of Lemma \ref{Lem:2.3}, $v(t)$ is strictly positive for $t \in [0,T]$. Indeed, arguing similarly as with Lemma \ref{Lem:2.3}, i.e. applying Feller’s non-attainability test to GMSI, it reads
\begin{align}\label{eq.71}
\mathcal{S}_{p}(x)&:=\kappa(\theta - x)+\frac{1}{2}\xi^2 x^{2p(x)}\left(p^{\prime}(x)\log(x)+\frac{p(x)}{x}\right)-\xi^2 x^{2p(x)}\left(p^{\prime}(x)\log(x)+\frac{p(x)}{x}\right)\\
& \geq \kappa(\theta - x)-\frac{1}{2}\xi^2 \left(|p^{\prime}(x)||\log(x)|x^{2p(x)}+x^{2p(x)-1}p(x)\right), \quad 0<x<\delta,  \nonumber
\end{align}
which yields
\begin{align}\label{eq.72}
\lim_{x \to 0^+}\mathcal{S}_{p}(x)& \geq \lim_{x \to 0^+} \left[\kappa(\theta - x)-\frac{1}{2}\xi^2 \left(|p^{\prime}(x)||\log(x)|x^{2p(x)}+x^{2p(x)-1}p(x)\right)\right]=\kappa \theta >0.
\end{align}
Therefore, the diffusion process $v(t)$ satisfying GMSI started in $(0,\infty)$ never hits zero a.s., and hence, GMSI is well-defined.\\

In the sequel, to analyze  moments dynamics, we will derive the formal equations for the evolution of the mean and variance for both models to highlight the structural differences.\\
Assume $v(t)$ ($v(0)>0$) be the solution to GM and $v_S(t)$ ($v_S(0)>0$) be the solution to GMSI, with their respective means $\mathbb{E}v(t)$ and $\mathbb{E}v_S(t)$, $t \in [0,T]$. Now, taking expectations and using the zero mean property of the It\^{o} integral yields
\begin{align}\label{eq.73}
\frac{d}{dt}\mathbb{E}v(t) &= \kappa(\theta -\mathbb{E}v(t)),
\end{align}
and
\begin{align}\label{eq.74}
\frac{d}{dt}\mathbb{E}v_S(t) &= \kappa(\theta -\mathbb{E}v_S(t)) + \mathbb{E} h(v_S(t)).
\end{align}
Depending on the specific form of $p(v)$ (increasing or decreasing) and the value of $v$, the sign of $\mathbb{E} h(v)$ is determined by the state-dependent function $ v \mapsto p^{\prime}(v) \log v+ \frac{p(v)}{v}=\frac{d}{dv}\left[p(v) \log v \right]$. That is,
\begin{equation}\label{eq.75}
\mathbb{E} h(v) \begin{cases} > 0 , & \text{if }\, p(v)\log v \text{ is strictly increasing}, \\ <0, & \text{if  }\, p(v)\log v \text{ is strictly decreasing}. \end{cases}
\end{equation}
Thus, GMSI is expected to produce a lower or higher mean than GM does. Notice that GMSI and GM can only produce the same mean if $p(v) \log v \equiv constant$, in which case the two models coincides. However, this is not the case under the hypotheses $(\mathbf{p_1})$-$(\mathbf{p_2})$. \\

Next, consider the second moments $\mathbb{E}v^2(t)$ and $\mathbb{E}v^2_S(t)$, $t \in [0,T]$, of GM and GMSI,  respectively.
Applying the It\^o formula for  GM and GMSI  yields
\begin{align}\label{eq.76}
dv^{2}(t) &= \left[2\kappa(\theta - v(t))v(t)+\xi^{2}v(t)^{2p(v(t))}\right]dt + 2\xi v(t)^{1+p(v(t))} dW(t),
\end{align}
and
\begin{align}\label{eq.77}
dv^{2}(t) &= \left[2\kappa(\theta - v(t))v(t)+\xi^{2}v(t)^{2p(v(t))}+2v(t)h(v(t))\right]dt + 2\xi v(t)^{1+p(v(t))} dW(t),
\end{align}
respectively. Taking expectations provides
\begin{align}\label{eq.78}
\frac{d}{dt}\mathbb{E}v^2(t) &= \mathbb{E}\left[2\kappa(\theta - v(t))v(t)+\xi^{2}v(t)^{2p(v(t))}\right],
\end{align}
and
\begin{align}\label{eq.79}
\frac{d}{dt}\mathbb{E}v^2_S(t) &= \mathbb{E}\left[2\kappa(\theta - v_S(t))v_S(t)+\xi^{2}v_S(t)^{2p(v_S(t))}\right]+\mathbb{E}\left[2v_S(t)h(v_S(t))\right].
\end{align}
The rate of change of the second moment for GMSI is augmented by the term $\mathbb{E}\left[2v_S(t)h(v_S(t))\right]$. This implies that the two interpretations lead to different dynamics for the second moment and, consequently, the variance.

\begin{remark}
The choice between It\^o and Stratonovich interpretations is not merely technical, rather, it has practical consequences for the behavior and analysis of stochastic models. Our generalized CIR model \eqref{eq.1} extends the classical Cox-Ingersoll-Ross (CIR) framework, commonly used in financial mathematics to model positive processes like interest rates or volatility. In this context, the It\^o interpretation  is standard, as it aligns with martingale-based methods for option pricing and ensures analytical tractability for deriving moment bounds and positivity under our assumptions $(\mathbf{p_1})$ and $(\mathbf{p_2})$. While the parametric noise (via $v(t)^{p(v(t))}$) introduces a state-dependent diffusion, the Itô framework remains appropriate and widely accepted for such financial models, as seen in standard CIR and related stochastic volatility models (e.g., Heston model); however, the existence of this parametric noise not necessitate Stratonovich interpretation in financial modeling since the Stratonovich dynamic lacks martingale property. We believe that the Stratonovich form may be more relevant for physical systems with parametric noise, such as in physical systems with multiplicative noise, where it better preserves symmetries like the chain rule.
\end{remark}

\section*{Data usage statement}
\noindent All data used in the numerical experiments are generated by Monte Carlo simulation, no proprietary or external datasets were used. All figures and distributions obtained purely from simulating the stochastic differential equations studied in the paper by using Python programming language.
\section*{Acknowledgments}
\noindent This work was supported by Athabasca University Research Incentive Account [140111 RIA].
\section*{Conflict of interest}
\noindent The author declares that he has no conflict of interest.
\section*{ORCID}
\noindent https://orcid.org/0000-0002-6001-627X

\end{document}